\documentclass[reqno]{amsart} 
\usepackage{amsfonts, amsmath, amsthm, amssymb, latexsym, graphicx}
\usepackage{color}

\newcommand{\R}{\mathbb{R}}

\theoremstyle{plain}

\newtheorem{theorem}{Theorem}[section]

\newtheorem{lemma}[theorem]{Lemma}
\newtheorem{proposition}[theorem]{Proposition}
\theoremstyle{definition}

\newtheorem{remark}[theorem]{Remark}
\theoremstyle{remark}

\newtheorem{step}[]{Step}
\numberwithin{equation}{section}

\title[Limit theorems
for  Jacobi  ensembles]{Limit theorems
	for  Jacobi  ensembles with large parameters}

\author{Kilian Hermann, Michael Voit}
\address{Fakult\"at Mathematik, Technische Universit\"at Dortmund,
	Vogelpothsweg 87,
	D-44221 Dortmund, Germany}
\email{kilian.hermann@math.tu-dortmund.de, michael.voit@math.tu-dortmund.de}


\begin{document}
\begin{abstract} Consider    $\beta$-Jacobi  ensembles with the distributions
	$$c_{k_1,k_2,k_3}\prod_{1\leq i< j \leq N}\left(x_j-x_i\right)^{k_3}\prod_{i=1}^N
	\left(1-x_i\right)^{\frac{k_1+k_2}{2}-\frac{1}{2}}\left(1+x_i\right)^{\frac{k_2}{2}-\frac{1}{2}} dx$$
	of the eigenvalues on the alcoves $A:=\{x\in\R^N| \> -1\leq x_1\leq ...\leq x_N\leq 1\}$.
	For  $(k_1,k_2,k_3)=\kappa\cdot (a,b,1)$ with $a,b>0$ fixed, we derive a central limit theorem
	for these distributions for  $\kappa\to\infty$.
	The drift and the  covariance matrix of the limit are expressed in terms of the zeros of 
	classical Jacobi polynomials.
	We also  determine the eigenvalues and eigenvectors of the covariance matrices.
	
	These results are related to corresponding limits for $\beta$-Hermite and $\beta$-Laguerre
	ensembles for $\beta\to\infty$ by Dumitriu and Edelman and by Voit.
\end{abstract}

\maketitle

\section{Introduction} 

We derive a central limit theorem (CLT) for  $\beta$-Jacobi random matrix ensembles for
fixed dimension $N$ where all parameters of the models tend to infinity. These  ensembles
are usually  described (see e.g. \cite{F, K, KN, M})
 via their joint eigenvalue distributions 
$\mu_k$ on the  alcoves 
 \begin{equation}\label{alcove-def}
   A:=\{x\in\R^N| \> -1\leq x_1\leq ...\leq x_N\leq 1\}\end{equation}
with the Lebesgue densities
\begin{equation}\label{joint-density}
c_{k_1,k_2,k_3}\prod_{1\leq i< j \leq N}\left(x_j-x_i\right)^{k_3}\prod_{i=1}^N
\left(1-x_i\right)^{\frac{k_1+k_2}{2}-\frac{1}{2}}\left(1+x_i\right)^{\frac{k_2}{2}-\frac{1}{2}} 
\end{equation}
with parameters $k:=(k_1,k_2,k_3)\in[0,\infty[^3$ and
a normalization $c_k$ which can be  determined via a  Selberg integral; see \cite{FW} for the background.

It is known from Kilip and Nenciu \cite{KN}
that all measures $\mu_k$ appear as joint distributions of the ordered eigenvalues of some tridiagonal random matrix models
similar to the tridiagonal models for $\beta$-Hermite and  $\beta$-Laguerre models of Dumitriu and Edelman \cite{DE1}.
Another matrix model in the Jacobi case is given in \cite{L}.

The tridiagonal models for $\beta$-Hermite and  $\beta$-Laguerre models of  \cite{DE1} are used in  \cite{DE2}
to derive limit theorems for $\beta\to\infty$. In particular, 
\cite{DE2} contains an  CLT where the covariance matrices $\Sigma$ of the limits are described 
in terms of the zeros of the $N$-th Hermite or Laguerre polynomial respectively.
Moreover, these CLTs were derived in \cite{V} directly where there
 formulas appear for the inverses  $\Sigma^{-1}$. In the present paper we
transfer the approach of  \cite{V}   to $\beta$-Jacobi ensembles.
For $k=(k_1,k_2,k_3)=\kappa\cdot(a,b,1)$, $N\in\mathbb N$, $a\ge0$, $b>0$ fixed, $\kappa\to\infty$,
we  prove an   CLT where the drift  and the 
inverse  $\Sigma^{-1}$ of the  covariance matrices are described in terms of the zeros of some Jacobi
polynomial  $P_N^{(\alpha,\beta)}$; see Theorem \ref{theoremCLT}.
Our CLT is closely related to the CLT A.1 in the appendix A of \cite{BG}. Moreover, our CLT
with its inverse covariance matrices is used in \cite{AHV} to compute the
covariance matrices themselves.  We expect that the results of the present paper can be used to derive limit results
for $\kappa\rightarrow\infty$ and then $N\rightarrow\infty$  as in \cite{AHV,GK} for Hermite ensembles.
 Further related
CLTs can be found in Proposition 2.3 of \cite{N} and in
 \cite{J, KN}.

We  mention that for all  $k:=(k_1,k_2,k_3)\in[0,\infty[^3$,
the measures $\mu_k$ on $A$ are the stationary distributions of so called 
$\beta$-Jacobi processes $(X_t^{k})_{t\ge0}$; see \cite{Dem}. These
processes are diffusions on $A$ with reflecting boundaries where the generators of the associated
Feller semigroups are  second order differential operators $D_{k}$
which  appear in the  Heckman-Opdam theory of 
hypergeometric functions associated with root systems; see \cite{HS}. The Heckman-Opdam Jacobi polynomials 
form multivariate systems of orthogonal polynomials with the  $\mu_k$ as orthogonality measures, and
 they are eigenfunctions of the   $D_{k}$. 
In the case of Hermite and Laguerre ensembles, 
the associated  diffusions are multivariate Bessel processes which appear
in the study of Calogero-Moser-Sutherland particle models \cite{DV, F}.
Limit theorems for the Bessel processes for large parameters were studied in this context in \cite{AKM1, AKM2, AV1, VW}.
We expect that similar results are  available for  $\beta$-Jacobi processes.

A comment about our  parameters $(k_1,k_2,k_3)$ which come from the
special functions associated with the root system $BC_N$;
 see \cite{HS, AV1, AV2, V, VW}.
In the random matrix
community  usually our $\kappa=k_3$ is denoted by  $\beta$.

This paper is organized as follows: In Section 2 we show that the measures $\mu_{\kappa\cdot(a,b,1)}$ tend to
some point measure $\delta_z$ for $\kappa\to\infty$ where the coordinates of  $z\in A$ consist
of the ordered zeros of the classical Jacobi polynomials 
$P_N^{(\alpha,\beta)}$ with $\alpha:=a+b-1>-1$ and $\beta=b-1>-1$.
This result is in principle known (see Section 6.7 of \cite{S},
Section  3.5 of \cite{I}, or Appendix A of \cite{BG}) and is needed for
 our CLT,  the main result of this paper.
 We shall state this CLT  in algebraic and trigonometric coordinates.
  Moreover we  discuss how our CLT is  related to the corresponding
CLTs for Hermite and Laguerre ensembles in \cite{DE2, V}.
Section 3 is then devoted to the proof of the CLT and  the eigenvalues and 
eigenvectors of the covariance matrices in  trigonometric coordinates.

\section{Central limit theorems in the freezing regime}

Consider  multiplicity parameters 
$k=(k_1,k_2,k_3)=\kappa\cdot(a,b,1)$ 
where we fix $a\ge0$, $b>0$. We study the limit $\kappa\to\infty$
of  the probability measures $\mu_k$ with densities (\ref{joint-density})
on  the alcove $A$ defined in (\ref{alcove-def}). 
 For this let $X_\kappa$ be  $\R^N$-valued random variables
with the distributions 
$$\mu_\kappa:=\mu_{\kappa\cdot(a,b,1)}.$$

As the $\mu_\kappa$ have Lebesgue-densities $f_\kappa$ of the form
\begin{equation}\label{product-form}
f_\kappa(x)= c_\kappa g(x) \phi(x)^\kappa 
\quad\quad\text{with}\quad\quad  c_\kappa:=c_{\kappa\cdot(a,b,1)}    
\end{equation}
on $A$ with suitable continuous functions $g,\phi$  and suitable constants $c_\kappa$,
we use the following well-known  Laplace method to obtain a first
limit law:

\begin{lemma}\label{factumLLN}
	Let $g,\phi:\R^N\rightarrow \R_+$ be continuous functions such that $\phi$ has a unique global
	maximum at $x_0\in\R^N$. If $g(x_0)>0$, and if $g\cdot \phi^\kappa\in L^1(\R^N, \lambda^N)$ for  $\kappa\ge1$, then
	the probability measures with the Lebesgue-densities
	\begin{align*}
	\frac{1}{\int_{\R^N}g(y)(\phi(y))^\kappa dy}\cdot g(x)(\phi(x))^\kappa
	\end{align*}
	tend weakly to  $\delta_{x_0}$.
\end{lemma}

This fact motivates us to analyze the function $\phi$  in (\ref{product-form}).
For this we use the  Jacobi polynomials  $(P_n^{(\alpha,\beta)})_{n\ge0}$ which are orthogonal polynomials
w.r.t.~the weights $$(1-x)^\alpha(1+x)^\beta \quad\text{on}\quad ]-1,1[$$
 for $\alpha,\beta>-1$.
For  details on these polynomials we refer to  \cite{S}.
We in particular need the following known facts on the ordered zeros 
$z_1\le \ldots\le z_N$ of  $P_N^{(\alpha,\beta)}$.

\begin{lemma}\label{lemmaMax} Let $a\ge0$, $b>0$, and
	 $\alpha:=a+b-1>-1$,  $\beta=b-1>-1$. Then the function
	\begin{align*}
	\phi(x):=\prod_{i<j}(x_j-x_i)\prod_{j=1}^N(1-x_j)^{\frac{a+b}{2}}(1+x_j)^{\frac{b}{2}}
	\end{align*}
	has a unique maximum on  $A$  at $z:=(z_1,...,z_N)\in A$.
	Moreover,
	\begin{equation}\label{lemma-formel1}
          \sum_{i=1,i\neq j}^{N}\frac{1}{z_j-z_i}+\frac{a+b}{2}\frac{1}{z_j-1}+\frac{b}{2}\frac{1}{z_j+1}=0
          \quad\text{for}\quad j=1,...,N,	\end{equation}
		\begin{equation}\label{lemma-formel2}
	 \phi(z)=2^{\frac{N}{2}(N+\alpha+\beta+1)}\prod_{j=1}^N
		j^{\frac{j}{2}}
		\frac{(\alpha+j)^{\frac{\alpha+j}{2}}(\beta+j)^{\frac{\beta+j}{2}}}{(N+\alpha+\beta+j)^{\frac{N+\alpha+\beta+j}{2}}},\end{equation}
\begin{equation}\label{prod 1-zi} \prod_{j=1}^N(1-z_j)=
  \frac{P_N^{(\alpha,\beta)}(1)}{l_N^{\alpha,\beta}}=2^{N}\prod_{j=1}^N\frac{\alpha+j}{N+\alpha+\beta+j},
  \end{equation}
and	\begin{equation}\label{prod 1+zi}	\prod_{j=1}^N(1+z_j)=\frac{P_N^{(\beta,\alpha)}(1)}{l_N^{\alpha,\beta}}
		=2^{N}\prod_{j=1}^N\frac{\beta+j}{N+\alpha+\beta+j}.\end{equation}
                \end{lemma}

\begin{proof}
  For the first  statement and (\ref{lemma-formel1})
  see Theorem 6.7.1 of \cite{S} or (3.5.4) of \cite{I}.
	The discriminant formula (\ref{lemma-formel2}) follows from (3.4.16) of \cite{I}.
  
Moreover,  as by (4.21.6) and (4.1.1) in \cite{S},
	\begin{align*} P_N^{(\alpha,\beta)}(x)=l_N^{\alpha,\beta}\prod_{j=1}^N(x-z_j) \quad\text{with}\quad
	l_N^{\alpha,\beta}=2^{-N}\prod_{j=1}^N\frac{N+\alpha+\beta+j}{j}\end{align*}
        and $P_N^{(\alpha,\beta)}(1)=\binom{N+\alpha}{N}$, we obtain   (\ref{prod 1-zi}). Finally, as
	$P_N^{(\alpha,\beta)}(-1)=(-1)^NP_N^{(\beta,\alpha)}(1)$ by (4.1.3) in \cite{S}, we get (\ref{prod 1+zi}).
        \end{proof}

 Lemmas  \ref{factumLLN} and \ref{lemmaMax} lead to the following limit theorem:

\begin{theorem}\label{LLN}
	Let $X_\kappa$ be random variables  as above. Let $z=(z_1,...,z_N)$ be the vector in the interior of $A$ which consists of the
	the ordered zeros of $P_N^{(\alpha,\beta)}$ with $\alpha,\beta$ as in  Lemma \ref{lemmaMax}. Then, for $\kappa\to\infty$ the $X_\kappa$ converge to $z$ in probability.
\end{theorem}

\begin{proof} Lemmas  \ref{factumLLN} and \ref{lemmaMax} imply that the distributions $\mu_\kappa$ of
  the  $X_\kappa$ tend weakly to $\delta_z$. 
	This fact is equivalent to the statement of the theorem.
\end{proof}

We now study the Jacobi ensemble in trigonometric coordinates
which fits to the theory of special functions associated with the root systems.
For this we define the probability measures $\tilde\mu_k$ on the trigonometric alcoves
$$\tilde A:=\{t\in\R^N|\frac{\pi}{2}\geq t_1\geq...\geq t_N\geq 0 \}$$
with the Lebesgue densities
\begin{equation}\label{density-joint-trig}
\tilde c_k\cdot \prod_{1\leq i< j \leq N}\left(\cos(2t_j)-\cos(2t_i)\right)^{k_3}
\prod_{i=1}^N\Bigl(\sin(t_i)^{k_1}\sin(2t_i)^{k_2}\Bigr)
\end{equation}
with a suitable  normalization $\tilde c_k>0$.
A short computation shows that the  measures $\mu_k$
on  $A$ with the densities (\ref{joint-density}) are the pushforward measures of the 
  $\tilde\mu_k$  under the transformation
\begin{align}\label{transformationtrigalgebraic}	
T: \tilde A\longrightarrow A, \quad T(t_1,\ldots,t_N):=(\cos(2t_1),\ldots, \cos(2t_N)).
\end{align}
%The Jacobi matrices $D$ of $T$ are diagonal matrices with
%\begin{align}\label{jacobimatrix}
%	(\tilde{D})_{i,i}(t_1,\ldots,t_N):=-2\sin(2t_i)%\text{ and } (D)_i:=-\frac{1}{2}\frac{1}{\sqrt{1-x_i^2}}
%\end{align} 

%If we use the Delta method for the central limit theorem for transformed random variables in Section 3.1 of \cite{vV},
%we readily obtain the following transformed CLT for the measures $\tilde\mu_k$ from  Theorem \ref{theoremCLT}:

%In this section we  postulate a central limit theorem for the random 
%variables $X_\kappa$ which improves the limit law \ref{LLN}. 
%We  proceed here similar to the CLTs in \cite{V} for $\beta$-Hermite and $\beta$-Laguerre ensembles. The main result is as follows:
Using this transformation, Theorem \ref{LLN} reads as follows:

\begin{theorem}\label{LLNtrig}
	Let $a\ge0$ and $b>0$.
	Let $\tilde X_\kappa$ be $\tilde A$-valued  random variables with the distributions 
	$\tilde\mu_{\kappa\cdot(a,b,1)}$ with the densities (\ref{density-joint-trig}) for $\kappa>0$.
	%Let $X_\kappa$ be random variables with distibutions as described in \ref{density} and \ref{product-form}. Let $z=(z_1,...,z_N)$ be the vector in the interior of $\tilde{A}$ which consists of the
	%the ordered zeros of $P_N^{(\alpha,\beta)}$ with $\alpha,\beta$ as in  Lemma \ref{lemmaMax}.
	 Then, for $\kappa\to\infty$ the $\tilde{X}_\kappa$ converge to $T^{-1}(z)=(\frac{1}{2}\arccos z_1, \ldots,\frac{1}{2}\arccos z_N)$ in probability.
\end{theorem}
We  now turn to a CLT for the random variables $\tilde X_\kappa$ in trigonometric form
which is the main result of this paper. It will be proved in Section 3.
 
 \begin{theorem}\label{theoremclttrig}
 	Let $a\ge0$ and $b>0$.
 	Let $\tilde X_\kappa$ be random variables with the distributions $\tilde{\mu}_\kappa$. 
 	Then $$\sqrt{\kappa}(\tilde X_\kappa-T^{-1}(z))$$
 	converges in distribution for $\kappa\rightarrow\infty$ to the centered $N-$dimensional normal distribution
        $N(0,\Sigma)$ with  covariance matrix $\tilde{\Sigma}$ whose inverse  
 	$\tilde{\Sigma}^{-1}=:\tilde{S}=(\tilde{s}_{i,j})_{i,j=1,...,N}$ satisfies
 	\begin{align*}
 	\tilde s_{i,j}=
 	\begin{cases}4\sum_{ l\ne j}
 	\frac{1-z_j^2}{(z_j-z_l)^2}+2(a+b)\frac{1+z_j}{1-z_j}+2b\frac{1-z_j}{1+z_j} &\textit{ for }i=j\\
 	\frac{-4\sqrt{(1-z_j^2)(1-z_i^2)}}{(z_i-z_j)^2}&\textit{ for }i\neq j
 	\end{cases}
 	\end{align*}
 	Furthermore the eigenvalues of $\tilde{\Sigma}^{-1}$ are simple and given by
 	$$\lambda_k=2k(2N+\alpha+\beta+1-k)>0 \quad\quad (k=1,\ldots,N).$$
 	Each  $\lambda_k$ has a eigenvector of the form
 	$$v_k:=\left(q_{k-1}(z_1)\sqrt{1-z_1^2},\ldots,q_{k-1}(z_N)\sqrt{1-z_N^2}\right)^T$$
 	for polynomials $q_{k-1}$ of order $k-1$ which are orthonormal w.r.t the discrete measure
 	\begin{align*}
 	\mu_{N,\alpha,\beta}:=(1-z_1^2)\delta_{z_1}+\ldots+(1-z_N^2)\delta_{z_N}
 	\end{align*}
 \end{theorem}

 This CLT can be transfered clearly into a CLT in algebraic coordinates. However, in these coordinates, the 
 eigenvalues and eigenvectors are more complicated:

\begin{theorem}\label{theoremCLT}
	Let $a\ge0$ and $b>0$.
	Let $X_\kappa$ be random variables with the distributions $\mu_\kappa$ as described in the beginning of this section. 
	Then $$\sqrt{\kappa}(X_\kappa-z)$$
	converges for $\kappa\rightarrow\infty$ to the centered $N-$dimensional normal distribution $N(0,\Sigma)$ with 
	 covariance matrix $\Sigma$ whose inverse  
	$\Sigma^{-1}=:S=(s_{i,j})_{i,j=1,...,N}$ is given by
	\begin{align*}
	s_{i,j}=
	\begin{cases}\sum_{l=1,\ldots,N; l\ne j}
	\frac{1}{(z_j-z_l)^2}+\frac{a+b}{2}\frac{1}{(1-z_j)^2}+\frac{b}{2}\frac{1}{(1+z_j)^2} &\textit{ for }i=j\\
	\frac{-1}{(z_i-z_j)^2}&\textit{ for }i\neq j
	\end{cases}
	\end{align*}
\end{theorem}

Our CLTs \ref{theoremclttrig} and \ref{theoremCLT} are closely related with the following determinantal formula
for the zeros of the Jacobi polynomials. It will be also proved in the next section.

\begin{proposition}\label{proposition-jacobi-det}
	For $N\in\mathbb{N}$ consider the ordered zeros $z_1\leq...\leq z_N$ of  $P_N^{(\alpha,\beta)}$ 
	with  $\alpha,\beta>-1$. Then the determinant of the matrix $S:=(s_{i,j})_{i,j=1,...,N}$ with
	\begin{align*}
	s_{i,j}=
	\begin{cases}\sum_{l=1,\ldots,N; l\ne j}
	\frac{1}{(z_j-z_l)^2}+\frac{\alpha+1}{2}\frac{1}{(1-z_j)^2}+\frac{\beta+1}{2}\frac{1}{(1+z_j)^2} &\text{ for }i=j\\
	\frac{-1}{(z_i-z_j)^2}&\text{ for }i\neq j
	\end{cases}
	\end{align*}
	satisfies
	\begin{align*}
	\det(S)=\frac{{N!}}{2^{3N}}\frac{{((N+\alpha+\beta+1)_N)^{3}}}{{(\alpha+1)_N(\beta+1)_N}}.
	\end{align*}
\end{proposition}

\begin{remark}
	Theorem \ref{theoremCLT} and Proposition \ref{proposition-jacobi-det} are closely related to corresponding results for
	Hermite and Laguerre ensembles in \cite{V}. Moreover, the distributions of Hermite and Laguerre ensembles
	may be seen as limits of Jacobi ensembles  for suitable limits for $\alpha=\beta\to\infty$ 
	(i.e., $a=0$ and $b\to\infty$) and $\alpha\to\infty$, $\beta>-1$ fixed (i.e.,  $a\to\infty$, $b>0$ fixed) respectively.
	These limits may be used to regard some results in \cite{V} as limits of 
	 Theorem \ref{theoremCLT} and Proposition \ref{proposition-jacobi-det}.
	
	We explain this in the Hermite case first:
	We fix $N$ and consider  $\alpha=\beta\to\infty$. It is well known (see Eq.~(5.6.3) of \cite{S}) that 
	\begin{equation}\label{limit-hermite-pol}
	\lim_{\alpha\to\infty}r_\alpha P_N^{(\alpha,\alpha) }(x/\sqrt\alpha)=c_N\cdot H_N(x)
	\end{equation}
	for the Hermite polynomial $H_N$ with some constants $C_N,r_\alpha >0$. 
	We now denote the ordered zeros of  $P_N^{(\alpha,\alpha) }$ by $z_1^{(\alpha)},\ldots, z_N^{(\alpha)}$, and 
	the ordered zeros of $H_N$ by  $z_1^H,\ldots, z_N^H$. We then have
	$\lim_{\alpha\to\infty}\sqrt\alpha\cdot z_j^{(\alpha)}= z_j^H$ for $j=1,\cdots,N$.
	We now insert these limits into the matrices $S^{(\alpha)}$ of  Theorem \ref{theoremCLT} and obtain
	\begin{equation}\label{lim-S-Hermite}\lim_{\alpha\to\infty} \frac{1}{\alpha} S^{(\alpha)} = S^H
	\end{equation}
	with the matrix $S^H=(s^H_{i,j})_{i,j=1,\ldots,N}$ with
	\begin{equation}\label{covariance-matrix-A}
	s^H_{i,j}:=\left\{ \begin{array}{r@{\quad\quad}l}  1+\sum_{l\ne i} (z_i^H-z_l^H)^{-2} & \text{for}\quad i=j \\
	-(z_i^H-z_j^H)^{-2} & \text{for}\quad i\ne j  \end{array}  \right.   
	\end{equation}
	which appears in the CLT for Hermite ensembles in  \cite{V}.
	Proposition \ref{proposition-jacobi-det} and (\ref{lim-S-Hermite}) now show that
	\begin{equation}\label{Hermite-det}
	\det(S^H)= \lim_{\alpha\to\infty} \frac{1}{\alpha^N}\det( S^{(\alpha)}) = N!.
	\end{equation}
	In summary, these limit results agree perfectly with the results in Section 2  of \cite{V}. 
\end{remark}

\begin{remark}
	In a similar way, the results in Section 3 of \cite{V} for Laguerre ensembles can be seen as limits of
	Theorem \ref{theoremCLT} and Proposition \ref{proposition-jacobi-det}. 
For  this we fix $b>0$, i.e. $\beta>-1$, and consider  $a\to\infty$, i.e. $\alpha\to\infty$.
	We
	recapitulate from (4.1.3) and  (5.3.4) of  \cite{S} that
	$$\lim_{\alpha\to\infty} P_N^{(\alpha,\beta) }(2x/\alpha-1)=(-1)^N \lim_{\alpha\to\infty} P_N^{(\beta,\alpha) }(1-2x/\alpha)
	=(-1)^N  L_N^{(\beta)}(x).$$
	We now denote the ordered zeros of  $P_N^{(\alpha,\beta) }$ by $z_1^{(\alpha)},\ldots, z_N^{(\alpha)}$, and 
	the ordered zeros of $ L_N^{(\beta)}$ by  $z_1^L,\ldots, z_N^L$. We then have
	\begin{equation}\label{limit-laguerre-ns}
	\lim_{\alpha\to\infty} \frac{\alpha}{2}(1+z_j^{(\alpha)})=z_j^L \quad\quad(j=1,\cdots,N).
	\end{equation}
	We now insert these limits into the matrices $S^{(\alpha)}$ of  Theorem \ref{theoremCLT} and obtain
	\begin{equation}\label{lim-S-Laguerre}
	\lim_{\alpha\to\infty} \frac{8}{\alpha^2} S^{(\alpha)} = S^L
	\end{equation}
	with the matrix $S^L=(s^L_{i,j})_{i,j=1,\ldots,N}$ with entries
	\begin{equation}\label{covariance-matrix-B}
	s^L_{i,j}:=\left\{ \begin{array}{r@{\quad\quad}l}   \frac{\beta+1}{(z_j^L)^2}+
	2\sum_{l\ne i} (z_i^L-z_l^L)^{-2} & \text{for}\quad i=j \\
	-2(z_i^L-z_j^L)^{-2} & \text{for}\quad i\ne j  \end{array}  \right.   
	\end{equation}
	Proposition  \ref{proposition-jacobi-det} and (\ref{lim-S-Laguerre}) now imply readily that
	\begin{equation}\label{Laguerre-det}
	\det(S^L)= \lim_{\alpha\to\infty} \frac{8^N}{\alpha^{2N}}\det( S^{(\alpha)}) = \frac{N!}{(\beta+1)_N}.
	\end{equation}
	The inverse limit covariance matrix  $S^L$ from (\ref{covariance-matrix-B}) and its determinant in (\ref{Laguerre-det})
	fits with the inverse limit covarianve matrix in the CLT 3.3 of \cite{V} and its determinant in Corollary 3.4 in  \cite{V}
	(for the starting point $0$ and time $t=1$ there).
	This connection is not  obvious as the Laguerre ensembles  in Section 3 of   \cite{V}
	are transformed, which is motivated by 
	the theory of multivariate Bessel processes. 
	To explain this connection, we recapitulate that in Section 3 of   \cite{V}, in the notation of the present paper,
	 random vectors $\tilde Y_{\beta+1,\alpha}$ are studied with  the Lebesgue densities
	\begin{equation}\label{laguerre-ensemble-old}
	\tilde c_{\beta+1,\alpha}^B  e^{-\|x\|^2/2}\prod_{i<j}(x_i^2-x_j^2)^{2\alpha}\cdot \prod_{i=1}^N x_i^{2(\beta+1)\alpha}
	\end{equation}
	on the Weyl chambers
	$$C_N^B:=\{x\in \mathbb R^N: \quad x_1\ge x_2\ge\ldots\ge x_N\ge0\}$$
	 with suitable  normalizations $\tilde c_{\beta+1,\alpha}^B>0$ for fixed $\beta>-1$ 
	and $\alpha\to\infty$.
	We now use the  zeros $z_1^{L}\ge\ldots\ge z_N^{L}$  of 
	$L_N^{(\beta)}$ as well as the vector
	\begin{equation}\label{def-r}
	r=(r_1,\ldots,r_N)\in C_N^B
	\quad\text{with}\quad  2(z_1^{L},\ldots, z_N^{L})= (r_1^2, \ldots, r_N^2).\end{equation}
	The CLT 3.3 and its Corollary 3.4 in  \cite{V} now state that
	$$\tilde Y_{\beta+1,\alpha} -  \sqrt{\alpha }\cdot r$$
	converges for $\alpha\to\infty$ to the centered $N$-dimensional distribution $N(0, (\tilde{S}^L)^{-1})$
	with the  covariance matrix $(\tilde{S}^L)^{-1}$ where the matrix $\tilde S^L=(\tilde s^L_{i,j})_{i,j=1,\ldots,N}$ satisfies
	\begin{equation}\label{covariance-matrix-B-old}
	\tilde s^L_{i,j}:=\left\{ \begin{array}{r@{\quad\quad}l}  1+ \frac{2(\beta+1)}{r_i^2}+2\sum_{l\ne i} (r_i-r_l)^{-2}+2\sum_{l\ne i} (r_i+r_l)^{-2} &
	\text{for}\quad i=j \\
	2(r_i+r_j)^{-2}  -2(r_i-r_j)^{-2} & \text{for}\quad i\ne j  \end{array}  \right.   
	\end{equation}
	and
	\begin{equation}\label{det-laguerre-old}
	\det\>\tilde  S^L= N!\cdot 2^N.
	\end{equation}
	
	It is clear that  the random vectors  $ Y_{\beta+1,\alpha}:=\tilde Y_{\beta+1,\alpha}^2/2$ (where the squares are taken in each component)
	have the Lebesgue densities
	\begin{equation}\label{laguerre-ensemble-new}
	c_{\beta+1,\alpha}^B e^{-(x_1+\ldots+x_N)}\prod_{i<j}(x_i-x_j)^{2\alpha}\cdot \prod_{i=1}^N x_i^{(\beta+1)\alpha-1/2}
	\end{equation}
	on $C_N^B$ with suitable normalizations $ c_{\beta+1,\alpha}^B>0$.
	The Delta-method for the central limit theorem of random variables,
	which are transformed under some smooth transform (see Section  3.1 of \cite{vV}) now implies that  
	$$\frac{1}{\sqrt\alpha}( Y_{\beta+1,\alpha}-\frac{\alpha}{2}r^2)=
	\frac{1}{2}(\tilde Y_{\beta+1,\alpha}- \sqrt\alpha \cdot r)\cdot
	\frac{\tilde Y_{\beta+1,\alpha}+ \sqrt{\alpha }\cdot r}{\sqrt\alpha}$$
	converges for $\alpha\to\infty$ to the centered $N$-dimensional distribution $N(0, S_L^{-1})$
	with transformed covariance matrix  $S_L^{-1}= D(\tilde S^L)^{-1}D$ with the diagonal matrix $D=diag(r_1,\ldots, r_N)$.
	If we use the equation in Lemma 3.1(2) of \cite{V} for the $r_i$, we obtain easily that the matrix 
	$S_L= D^{-1}\tilde S^LD^{-1}$ is equal to the matrix $S^L$ in (\ref{covariance-matrix-B}).
	Moreover, (\ref{limit-laguerre-ns}) and (\ref{prod 1+zi})  yield that
	\begin{equation}\label{prod-ns-Laguerre}
	\prod_{i=1}^N z_i^L=(\beta+1)_N;
	\end{equation}
	see also  (5.1.7) and (5.1.8) in \cite{S}.
	(\ref{prod-ns-Laguerre}) and (\ref{def-r}) now lead to
	$$\det\>S_L = \frac{1}{2^N (\beta+1)_N}\det\>\tilde S^L= \frac{N!}{ (\beta+1)_N}= \det\>S^L.$$
	These results fit to  (\ref{covariance-matrix-B}) and  (\ref{Laguerre-det}) as claimed.
\end{remark}

 The eigenvalues and eigenvectors of the 
	 $S^H$ and $\tilde S^L$   were determined in \cite{AV2} explicitely.
	On the other hand, it is more complicated to determine  the eigenvectors and eigenvalues of
	the matrix $S^L$ for the
	Laguerre ensembles (\ref{laguerre-ensemble-new}). Therefore, the difficulty of finding 
	the eigenvectors and eigenvalues  depends heavily on  the parametrization 
	of the random matrix ensembles.

        \section{Proof of the main results}\label{proofsofmaintheorem}
        
 In this section we  prove the CLTs \ref{theoremclttrig} and \ref{theoremCLT} and Proposition \ref{proposition-jacobi-det}.
 The proofs are divided into several parts.
 In the first step we derive a restricted version of
 Theorem \ref{theoremCLT}, where we shall only get vague  instead of weak convergence.

\begin{step}
  %The proof of this step is elementary, but quite technical.
	The representation (\ref{joint-density}) of the densities $f_{\kappa}$ of the variables $X_\kappa$ implies that
	the random variables $\sqrt{\kappa}(X_\kappa-z)$ have the Lebesgue densities
	\begin{align}
	\tilde f_\kappa(x):=&\frac{1}{\kappa^\frac{N}{2}}f_{\kappa}\left(\frac{x}{\sqrt{\kappa}}+z\right)\\
	=&\frac{c_\kappa}{\kappa^\frac{N}{2}}\prod_{j=1}^N\frac{1}{(1-(\frac{x_j}{\sqrt{\kappa}}+z_j)^2)^{\frac{1}{2}}}\times\notag\\&\> \times\left(\prod_{i<j}{\left(\frac{x_j-x_i}{\sqrt{\kappa}}+z_j-z_i\right)}
	\prod_{j=1}^N\left(1-\frac{x_j}{\sqrt{\kappa}}-z_j\right)^{\frac{a+b}{2}}\left(1+\frac{x_j}{\sqrt{\kappa}}+z_j\right)^{\frac{b}{2}}\right)^\kappa\notag
	\end{align}
	on  $\sqrt{\kappa}(A-z)$ and zero elsewhere. We  split this formula into two parts
	\begin{align}\label{splitfk}
	\tilde f_\kappa(x)= C_\kappa h_\kappa(x)
	\end{align}
	where $h_\kappa$ depends on $x$ and  $C_\kappa$ is constant w.r.t. $x$.
	More precisely, we put
	\begin{align}\label{cktilde}
	C_\kappa:=\frac{c_\kappa}{\kappa^\frac{N}{2}}\prod_{j=1}^N\frac{1}{(1-z_j^2)^\frac{1}{2}}\cdot
	\left(\prod_{i<j}(z_j-z_i)\prod_{j=1}^N((1-z_j)^{\frac{a+b}{2}}(1+z_j)^\frac{b}{2})\right)^\kappa
	\end{align}
	and
	\begin{align}\label{defhk}
	h_\kappa(x)&:=
	\prod_{j=1}^N\frac{(1-z_j^2)^\frac{1}{2}}{(1-(\frac{x_j}{\sqrt{\kappa}}+z_j)^2)^{\frac{1}{2}}}\times\\
	&\times\left(\prod_{i<j}\left(1+\frac{x_j-x_i}{\sqrt{\kappa}(z_j-z_i)}\right)
	\prod_{j=1}^N\left(1-\frac{x_j}{\sqrt{\kappa}(1-z_j)}\right)^\frac{a+b}{2}\left(1+\frac{x_j}{\sqrt{\kappa}(1+z_j)}\right)^\frac{b}{2}\right)^\kappa.\notag
	\end{align}
	
	We first investigate $C_\kappa$. We here first 
	focus on the constants $c_\kappa$ in (\ref{joint-density}) and recapitulate from \cite{FW} the Selberg integral
\begin{align}\label{selberg-integral}
\int_{[0,1]^N}&\prod_{i<j}|x_i-x_j|^{2\rho}\prod_{i=1}^N(1-x_i)^{\nu-1} x_i^{\mu-1}dx\notag\\
&=\prod_{j=1}^N\frac{\Gamma(1+j\rho)}{\Gamma(1+\rho)}\frac{\Gamma(\mu+(j-1)\rho)\Gamma(\nu+(j-1)\rho)}{\Gamma(\mu+\nu+(N+j-2)\rho)}
\end{align}
for $\mu,\nu,\rho>0$. The  substitution
	$x_i= 2y_i -1$ ($i=1,\ldots,N$) then yields
	\begin{align}\label{bevorestirling}
	&\frac{1}{c_\kappa}=\int_{A}\prod_{1\leq i< j \leq N}\left(x_j-x_i\right)^{\kappa}\prod_{i=1}^N\left(1-x_i\right)^{\frac{\kappa(a+b)}{2}-\frac{1}{2}}\left(1+x_i\right)^{\frac{\kappa b}{2}-\frac{1}{2}}dx\notag\\
	&=\frac{1}{N!}\int_{[-1,1]^N}\prod_{1\leq i< j \leq N}\left|x_j-x_i\right|^{\kappa}\prod_{i=1}^N\left(1-x_i\right)^{\frac{\kappa(a+b)}{2}-\frac{1}{2}}\left(1+x_i\right)^{\frac{\kappa b}{2}-\frac{1}{2}}dx\notag\\
	&=\frac{1}{N!}\int_{[0,1]^N}\prod_{1\leq i< j \leq N}(2\left|y_j-y_i\right|)^{\kappa}\prod_{i=1}^N\left((2\left(1-y_i\right))^{\frac{\kappa(a+b)}{2}-\frac{1}{2}}\left(2y_i\right)^{\frac{\kappa b}{2}-\frac{1}{2}}2\right)dy\notag\\
	&=\frac{2^{\frac{\kappa N}{2}(N-1+a+2b)}}{N!}\int_{[0,1]^N}\prod_{1\leq i< j \leq N}\left|y_j-y_i\right|^{\kappa}\prod_{i=1}^N\left(1-y_i\right)^{\frac{\kappa(a+b)}{2}-\frac{1}{2}}\left(y_i\right)^{\frac{\kappa b}{2}-\frac{1}{2}}dy\notag\\
	&=\frac{2^{\frac{\kappa N}{2}(N-1+a+2b)}}{N!}\prod_{j=1}^N\frac{\Gamma(1+j\frac{\kappa}{2})}{\Gamma(1+\frac{\kappa}{2})}\frac{\Gamma(\frac{\kappa b}{2}+\frac{1}{2}+(j-1)\frac{\kappa}{2})\Gamma(\frac{\kappa b+\kappa a}{2}+\frac{1}{2}+(j-1)\frac{\kappa}{2})}{\Gamma(\frac{\kappa b}{2}+\frac{1}{2}+\frac{\kappa b+\kappa a}{2}+\frac{1}{2}+(N+j-2)\frac{\kappa}{2})}\notag\\
	&=\frac{2^{\frac{\kappa N}{2}(N+\alpha+\beta+1)}}{N!}\prod_{j=1}^N\frac{\Gamma(1+j\frac{\kappa }{2})}{\Gamma(1+\frac{\kappa }{2})}
	\frac{\Gamma\left(\frac{\kappa (\beta+j)}{2}+\frac{1}{2}\right)\Gamma\left(\frac{\kappa (\alpha+j)}{2}+\frac{1}{2}\right)}{\Gamma\left(\frac{\kappa (N+\alpha+\beta+j)}{2}+1\right)}
	\end{align} 
	where the notation $\alpha=a+b-1$ and $\beta=b-1$ from Lemma \ref{lemmaMax} was used.
	In order to study the limit behavior of (\ref{bevorestirling}) for $\kappa\to\infty$, we use the notation
	$$f(x)\sim g(x):\iff \lim_{x\rightarrow\infty}\frac{f(x)}{g(x)}=1.$$ 
	We also recapitulate Stirling's formula and two of its well-known consequences:
	\begin{align}
	\Gamma(1+x)=x\Gamma(x)&\sim\sqrt{2\pi x}\left(\frac{x}{e}\right)^x,\label{stirling1}\\
	\frac{\Gamma(\frac{1}{2}+x)}{\Gamma(1+x)}&\sim x^{\frac{1}{2}-1}=\frac{1}{\sqrt{x}},\quad 
	\Gamma(\frac{1}{2}+x)\sim\sqrt{2\pi}\left(\frac{x}{e}\right)^x\label{stirlig1halb}.
	\end{align} 
	We now apply these formulas to  (\ref{bevorestirling}). For this we first observe that (\ref{stirling1})
	leads to
	\begin{align}
	\prod_{j=1}^N\frac{\Gamma(1+j\frac{\kappa}{2})}{\Gamma(1+\frac{\kappa}{2})}
	&\sim\prod_{j=1}^N\frac{\sqrt{\pi j \kappa}\left(\frac{j\kappa}{2e}\right)^{\frac{j\kappa}{2}}}{\sqrt{\pi \kappa}\left(\frac{\kappa}{2e}\right)^{\frac{\kappa}{2}}}
	=\prod_{j=1}^Nj^{\frac{j\kappa+1}{2}}\left(\frac{\kappa}{2e}\right)^{\frac{\kappa}{2}(j-1)}\notag\\
	&=\sqrt{N!}\left(\frac{\kappa}{2e}\right)^{\frac{\kappa}{2}\frac{N(N-1)}{2}}
	\prod_{j=1}^N j^{j\frac{\kappa}{2}}.\label{Anwenungstirling1}
	\end{align}
	For the second part of (\ref{bevorestirling}) we use (\ref{stirling1}) and (\ref{stirlig1halb}) and get
	$$\Gamma\left(\frac{\kappa(\rho+j)}{2}+\frac{1}{2}\right)
        \sim\sqrt{2\pi}\left(\frac{\kappa(\rho+j)}{2e}\right)^{\frac{\kappa(\rho+j)}{2}}	\quad\quad(\rho=\alpha,\beta), \quad \text{and}$$       
\begin{align*}	&\Gamma\left(\frac{\kappa(N+\alpha+\beta+j)}{2}+1\right)\sim\\
  &\quad\quad \sim\sqrt{\pi \kappa(N+\alpha+\beta+j)}
  \left(\frac{\kappa(N+\alpha+\beta+j)}{2e}\right)^{\frac{\kappa(N+\alpha+\beta+j)}{2}}.
	\end{align*}
	These results lead to
	\begin{align}\label{Anwendundstirling2}
	&\prod_{j=1}^N\frac{\Gamma\left(\frac{\kappa(\beta+j)}{2}+\frac{1}{2}\right)\Gamma\left(\frac{\kappa(\alpha+j)}{2}+\frac{1}{2}\right)}{\Gamma\left(\frac{\kappa(N+\alpha+\beta+j)}{2}+1\right)}\\
	\sim&\prod_{j=1}^N\frac{2\sqrt{\pi}}{\sqrt{\kappa(N+\alpha+\beta+j)}}\left(\frac{2e}{\kappa}\right)^{\frac{\kappa}{2}(N-j)}\left(\frac{(\alpha+j)^{\frac{\alpha+j}{2}}(\beta+j)^{\frac{\beta+j}{2}}}{(N+\alpha+\beta+j)^{\frac{N+\alpha+\beta+j}{2}}}\right)^{\kappa}\notag\\
	=&\frac{2^N\pi^\frac{N}{2}}{\kappa^\frac{N}{2}}\left(\frac{2e}{\kappa}\right)^{\frac{\kappa}{2}\frac{N(N-1)}{2} }\prod_{j=1}^{N}\frac{1}{\sqrt{N+\alpha+\beta+j}}\left(\frac{(\alpha+j)^{\frac{\alpha+j}{2}}(\beta+j)^{\frac{\beta+j}{2}}}{(N+\alpha+\beta+j)^{\frac{N+\alpha+\beta+j}{2}}}\right)^{\kappa}\notag\\
	=&\frac{2^N\pi^\frac{N}{2}}{\kappa^\frac{N}{2}\sqrt{(N+\alpha+\beta+1)_{N}}}\left(\frac{2e}{\kappa}\right)^{\frac{\kappa}{2}\frac{N(N-1)}{2}}\left(\prod_{j=1}^N\frac{(\alpha+j)^{\frac{\alpha+j}{2}}(\beta+j)^{\frac{\beta+j}{2}}}{(N+\alpha+\beta+j)^{\frac{N+\alpha+\beta+j}{2}}}\right)^{\kappa}.
	\notag
	\end{align}
	Combining (\ref{Anwenungstirling1}) and (\ref{Anwendundstirling2}), we obtain
	\begin{align*}
	&\prod_{j=1}^N\frac{\Gamma(1+j\frac{\kappa}{2})}{\Gamma(1+\frac{\kappa}{2})}
	\frac{\Gamma\left(\frac{\kappa(\beta+j)}{2}+\frac{1}{2}\right)\Gamma\left(\frac{\kappa(\alpha+j)}{2}+\frac{1}{2}\right)}{\Gamma\left(\frac{\kappa(N+\alpha+\beta+j)}{2}+1\right)}\sim\\
	&\quad\sim\frac{\sqrt{N!}2^N\pi^\frac{N}{2}}{\kappa^\frac{N}{2}\sqrt{(N+\alpha+\beta+1)_N}}\left(\prod_{j=1}^Nj^\frac{j}{2}\frac{(\alpha+j)^\frac{\alpha+j}{2}(\beta+j)^\frac{\beta+j}{2}}{(N+\alpha+\beta+j)^\frac{N+\alpha+\beta+j}{2}}\right)^\kappa.
	\end{align*}
	Finally, if we apply this to (\ref{bevorestirling}), we see that ${1}/c_\kappa$ behaves like
	\begin{equation}\label{ck-final}
	\frac{2^{\frac{\kappa N}{2}(N+\alpha+\beta+1)}2^N\pi^\frac{N}{2}}{\sqrt{N!}\kappa^\frac{N}{2}\sqrt{(N+\alpha+\beta+1)_{N}}}\left(\prod_{j=1}^Nj^\frac{j}{2}\frac{(\alpha+j)^\frac{\alpha+j}{2}(\beta+j)^\frac{\beta+j}{2}}{(N+\alpha+\beta+j)^\frac{N+\alpha+\beta+j}{2}}\right)^\kappa.
	\end{equation}
	Having this limit of $c_\kappa$ in mind, we now determine the asymptotics  
	of  $C_\kappa$ in (\ref{cktilde}). For this we use (\ref{lemma-formel2}) with the function $\phi$ there
	as well as
	(\ref{prod 1+zi}), (\ref{prod 1-zi}), and (\ref{ck-final}). Using the Pochhammer symbol
	$(x)_N:= x(x+1)\cdots(x+N-1),$ we  get
	\begin{align*}
	C_\kappa &=\frac{c_\kappa}{\kappa^\frac{N}{2}}\left(\phi(z)\right)^\kappa\prod_{j=1}^N\frac{1}{(1-z_j^2)^\frac{1}{2}}
	=\frac{c_\kappa}{\kappa^\frac{N}{2}}\prod_{j=1}^N\frac{1}{(1-z_j)^\frac{1}{2}(1+z_j)^\frac{1}{2}}\left(\phi(z)\right)^\kappa\\
	&=\frac{c_\kappa2^{-N}}{\kappa^\frac{N}{2}}2^{\frac{\kappa N}{2}(N+\alpha+\beta+1)}\prod_{j=1}^N
	\left(j^{\frac{j}{2}}
	\frac{(j+\alpha)^{\frac{j+\alpha}{2}}(j+\beta)^{\frac{j+\beta}{2}}}{\left(N+\alpha+\beta+j\right)^{\frac{N+\alpha+\beta+j}{2}}}\right)^\kappa
	\frac{N+\alpha+\beta+j}{\sqrt{(\alpha+j)(\beta+j)}}\\
	&\sim\frac{\sqrt{N!}}{2^{2N}\pi^\frac{N}{2}}\frac{{((N+\alpha+\beta+1)_N)^{3/2}}}{\sqrt{{(\alpha+1)_N(\beta+1)_N}}}.
	\end{align*}
	In summary,
	\begin{align}\label{limcktilde}
	\lim_{\kappa\rightarrow\infty}C_\kappa=\frac{\sqrt{N!}}{2^{2N}\pi^\frac{N}{2}}
	\frac{{((N+\alpha+\beta+1)_N)^{3/2}}}{\sqrt{{(\alpha+1)_N(\beta+1)_N}}}.
	\end{align}
	
	We next turn to an asymptotics of  $h_\kappa(x)$  in (\ref{defhk}).
	We  first  observe  that
	\begin{equation}\label{converrterm}
	\prod_{j=1}^N\frac{(1-z_j^2)^\frac{1}{2}}{(1-(\frac{x_j}{\sqrt{\kappa}}+z_j)^2)^{\frac{1}{2}}}\longrightarrow1 \quad(\kappa\to\infty).
	\end{equation}
	 Hence, this factor  can  be ignored.
	It will be convenient to write the further factor $\tilde h_\kappa(x)$ of  $h_\kappa(x)$ in 
	the second line of (\ref{defhk}) as
	$\tilde h_\kappa(x)=\exp(\log(\tilde h_\kappa(x))).$
	We now have to investigate the term
	\begin{align}\label{valuehk}
	&\exp\Bigg(\kappa\Bigg(\sum_{i<j}\log\bigg(1+\frac{x_j-x_i}{\sqrt{\kappa}(z_j-z_i)}\bigg)+\\
	&+\frac{a+b}{2}\sum_{j=1}^N\log\left(1-\frac{x_j}{\sqrt{\kappa}(1-z_j)}\right)+\frac{b}{2}\sum_{j=1}^N\log\left(1+\frac{x_j}{\sqrt{\kappa}(1+z_j)}\right)\Bigg)\Bigg).\notag
	\end{align}
	We now apply Taylor's formula to all logarithms, i.e.,  for large $\kappa$, 
	\begin{align*}
	\log\left(1+\frac{x_j-x_i}{\sqrt{\kappa}(z_j-z_i)}\right)&=\frac{x_j-x_i}{\sqrt{\kappa}(z_j-z_i)}-\frac{(x_j-x_i)^2}{2\kappa(z_j-z_i)^2}+O(\kappa^{-\frac{3}{2}})\\
	\log\left(1- \frac{x_j}{\sqrt{\kappa}(1-z_j)}\right)&=\frac{-x_j}{\sqrt{\kappa}(1-z_j)}-\frac{x_i^2}{2\kappa(1-z_j)^2}+O(\kappa^{-\frac{3}{2}})\\
	\log\left(1+\frac{x_j}{\sqrt{\kappa}(1+z_j)}\right)&=\frac{x_j}{\sqrt{\kappa}(1+z_j)}-\frac{x_j^2}{2\kappa(1+z_j)^2}+O(\kappa^{-\frac{3}{2}}).
	\end{align*}
	By (\ref{lemma-formel1}), 
	\begin{align}\label{Nullterm}
	&\sum_{i<j}\frac{\sqrt{\kappa}(x_j-x_i)}{(z_j-z_i)}-\frac{a+b}{2}\sum_{j=1}^N\frac{\sqrt{\kappa}x_j}{1-z_j}
	+\frac{b}{2}\sum_{j=1}^N\frac{\sqrt{\kappa}x_j}{1+z_j}\\
	=&\sqrt{\kappa}\sum_{j=1}^Nx_j\left(\sum_{i=1,j\neq i}^N\frac{1}{z_j-z_i}-\frac{a+b}{2}\frac{1}{1-z_j}+\frac{b}{2}\frac{1}{1+z_j}\right)=0\notag
	\end{align}
	and therefore (\ref{valuehk}) turns into
	\begin{align*}
	\exp\left(-\frac{1}{2}\bigg(\sum_{i<j}\frac{(x_i-x_j)^2}{(z_j-z_i)^2}+\frac{a+b}{2}\sum_{j=1}^N\frac{x_j^2}{(1-z_j)^2}
	+\frac{b}{2}\sum_{j=1}^N\frac{x_j^2}{(1+z_j)^2}+O(\kappa^{-\frac{1}{2}})\bigg)\right).
	\end{align*}
	If we combine this with (\ref{converrterm}) we get
	\begin{align}\label{limhk}
	&\lim_{\kappa\rightarrow\infty}h_\kappa(x)
	\\
	&\quad=\exp\left(-\frac{1}{2}\bigg(\sum_{i<j}\frac{(x_i-x_j)^2}{(z_j-z_i)^2}+\frac{a+b}{2}\sum_{j=1}^N\frac{x_j^2}{(1-z_j)^2}
	+\frac{b}{2}\sum_{j=1}^N\frac{x_j^2}{(1+z_j)^2}\bigg)\right).
	\notag\end{align}
	Now let $f\in C_c(\R^N)$ be a  continuous function with compact support.
	From (\ref{splitfk}), (\ref{limcktilde}), (\ref{limhk}) and 
	dominated convergence we get
	\begin{align}\label{convindistr}
	&\lim_{\kappa\rightarrow\infty}\int_{\sqrt{\kappa}(A-z)} f(x)\tilde f_\kappa(x)dx
	=\lim_{\kappa\rightarrow\infty}C_\kappa\int_{\R^N}{\bf 1}_{\sqrt{\kappa}(A-z)}(x)f(x)h_\kappa(x)dx\\
	&=\frac{\sqrt{N!}}{2^{2N}\pi^\frac{N}{2}}\sqrt{\frac{{(N+\alpha+\beta+1)^{3}_N}}{{(\alpha+1)_N(\beta+1)_N}}}
	\int_{\R^N} f(x)
	\notag\\
	&\quad\times
	\exp\left(-\frac{1}{2}\bigg(\sum_{i<j}\frac{(x_i-x_j)^2}{(z_j-z_i)^2}+\frac{a+b}{2}\sum_{j=1}^N\frac{x_j^2}{(1-z_j)^2}
	+\frac{b}{2}\sum_{j=1}^N\frac{x_j^2}{(1+z_j)^2}\bigg)\right)dx\notag
	\end{align}
	We briefly check that  we can interchange the limit with integration in (\ref{convindistr}) by dominated convergence.
	For this we determine an integrable upper bound for $C_\kappa{\bf 1}_{\sqrt{\kappa}(A-z)}(x)|f(x)|h_\kappa(x)$.
	We  first observe that by (\ref{converrterm}), $f\in C_c(\R^N)$ and a short calculation,
	we find constants $C,\kappa_0>0$  such that
	for all $\kappa\ge\kappa_0 $ and  $x\in\R^N$,
	\begin{equation}\label{uniform-est-remainder}
	{\bf 1}_{\sqrt{\kappa}(A-z)}(x)|f(x)|
	\prod_{j=1}^N\frac{(1-z_j^2)^\frac{1}{2}}{(1-(\frac{x_j}{\sqrt{\kappa}}+z_j)^2)^{\frac{1}{2}}}\leq C
	\end{equation}
	holds. For the remaining factors we again use the Taylor expansion of $\log(1+x)$. 
	Here the Lagrange  remainder shows that
	\begin{align*}
	\log\left(1+\frac{x_j-x_i}{\sqrt{\kappa}(z_j-z_i)}\right)&=\frac{x_j-x_i}{\sqrt{\kappa}(z_j-z_i)}-\frac{(x_j-x_i)^2}{2\kappa(z_j-z_i)^2}w_{i,j},\\
	\log\left(1- \frac{x_j}{\sqrt{\kappa}(1-z_j)}\right)&=\frac{-x_j}{\sqrt{\kappa}(1-z_j)}-\frac{x_i^2}{2\kappa(1-z_j)^2}w_{j}^-,\\
	\log\left(1+\frac{x_j}{\sqrt{\kappa}(1+z_j)}\right)&=\frac{x_j}{\sqrt{\kappa}(1+z_j)}-\frac{x_j^2}{2\kappa(1+z_j)^2}w_{j}^+
	\end{align*}
	with $w_{i,j},w_j^+w_j^-\in(0,1)$ for $i,j=1,...,N$. If we set 
	$$w:=\min\{w_{i,j},w_j^-,w_j^+|i,j=1,...,N\}\in(0,1),$$
	we get 
	\begin{align*}
	&{\bf 1}_{\sqrt{\kappa}(A-z)}(x)\cdot|f(x)|\cdot h_\kappa(x)\leq
	\\ &\quad\leq C
	\exp\left(-w\frac{1}{2}\bigg(\sum_{i<j}\frac{(x_i-x_j)^2}{(z_j-z_i)^2}+\frac{a+b}{2}\sum_{j=1}^N\frac{x_j^2}{(1-z_j)^2}
	+\frac{b}{2}\sum_{j=1}^N\frac{x_j^2}{(1+z_j)^2}\bigg)\right).
	\end{align*}
	This and (\ref{uniform-est-remainder}) show
	that  dominated convergence in (\ref{convindistr}) is  available.
	Eq.~(\ref{convindistr}) means that $\sqrt{\kappa}(X_\kappa-z)$ converges vaguely
	to the  measure with the density
	\begin{align}\label{Limit}
	\frac{\sqrt{N!}}{2^{2N}\pi^\frac{N}{2}}\sqrt{\frac{{((N+\alpha+\beta+1)_N)^{3}}}{{(\alpha+1)_N(\beta+1)_N}}}\exp\left(-\frac{1}{2}x^\top\Sigma^{-1} x\right),
	\end{align}
	where $\Sigma^{-1}=(s_{i,j})_{i,j=1,...,N}$ is given as in Theorem \ref{theoremCLT}.
	As a vague limit of probability measures, this measure is a sub-probability measure. Moreover, this measure is
	 the normal distribution  in Theorem \ref{theoremCLT}
	 possibly up to the correct normalization constant.
         We shall postpone  the normalization to the end of this section.
\end{step}

In the next step we determine the eigenvectors and eigenvalues  of the  matrix $\tilde S$ in 
Theorem \ref{theoremclttrig} for $a\ge0$, $b>0$:

\begin{proposition}\label{ev-ew} 
	The matrix $\tilde{S}$ has the simple
	 eigenvalues
	$$\lambda_k=2k(2N+\alpha+\beta+1-k)>0 \quad\quad (k=1,\ldots,N).$$
	Moreover, each  $\lambda_k$ has a eigenvector of the form
	$$v_k:=\left(q_{k-1}(z_1)\sqrt{1-z_1^2},\ldots,q_{k-1}(z_N)\sqrt{1-z_N^2}\right)^T$$
	for polynomials $q_{k-1}$ of order $k-1$ which are orthonormal w.r.t the discrete measure
	\begin{align*}
		\mu_{N,\alpha,\beta}:=(1-z_1^2)\delta_{z_1}+\ldots+(1-z_N^2)\delta_{z_N}
	\end{align*}
\end{proposition}

The proof uses induction on $k$. For $k=1$ we have:

\begin{lemma}\label{ev1}
	The vector $v_1:=(\sqrt{1-z_1^2},\ldots,\sqrt{1-z_N^2})^T$ is an  eigenvector of  $\tilde S$ associated with the 
	eigenvalue $\lambda_1$.
\end{lemma}

\begin{proof} By the definition of  $\tilde S$, the $i$-th component ($i=1,\ldots,N$) of $\tilde Sv_1$ is given by
\begin{align}
	(\tilde Sv_1)_i =& 4\sum_{ l\ne i} \frac{1-z_i^2}{(z_i-z_l)^2}\sqrt{1-z_i^2}+ 2(a+b)\frac{1+z_i}{1-z_i}\sqrt{1-z_i^2}\notag\\
	& + 2b\frac{1-z_i}{1+z_i}\sqrt{1-z_i^2} - 4\sum_{ l\ne i} \frac{(1-z_l^2)\sqrt{1-z_i^2}}{(z_i-z_l)^2} \notag\\
	=& 4\sqrt{1-z_i^2}\left( \sum_{ l\ne i} \frac{z_l^2- z_i^2}{(z_i-z_l)^2}+ \frac{a+b}{2}\frac{1+z_i}{1-z_i}
	+ \frac{b}{2}\frac{1-z_i}{1+z_i}\right).\notag
\end{align}
Hence,
	\begin{align}
	(\tilde Sv_1)_i =& 4\sum_{ l\ne i} \frac{1-z_i^2}{(z_i-z_l)^2}\sqrt{1-z_i^2}+ 2(a+b)\frac{1+z_i}{1-z_i}\sqrt{1-z_i^2}\\
	& + 2b\frac{1-z_i}{1+z_i}\sqrt{1-z_i^2} - 4\sum_{ l\ne i} \frac{(1-z_l^2)\sqrt{1-z_i^2}}{(z_i-z_l)^2} \notag\\
	=& 4\sqrt{1-z_i^2}\left( \sum_{ l\ne i} \frac{z_l^2- z_i^2}{(z_i-z_l)^2}+ \frac{a+b}{2}\frac{1+z_i}{1-z_i}
	+ \frac{b}{2}\frac{1-z_i}{1+z_i}\right)\notag\\
	=& 4\sqrt{1-z_i^2}\left( \sum_{ l\ne i} \frac{-2z_i + (z_i-z_l)}{z_i-z_l}+ \frac{a+b}{2}\frac{1+z_i}{1-z_i}
	+ \frac{b}{2}\frac{1-z_i}{1+z_i}\right)\notag\\
	=& 4\sqrt{1-z_i^2}\left((N-1)-2z_i \sum_{ l\ne i} \frac{1}{z_i-z_l}+ \frac{a+b}{2}\frac{1+z_i}{1-z_i}
	+ \frac{b}{2}\frac{1-z_i}{1+z_i}\right)\notag.
	\end{align}
	 (\ref{lemma-formel1}) now leads to
	$$(\tilde Sv_1)_i = 4\sqrt{1-z_i^2}\left((N-1)+\frac{a+b}{2}+ \frac{b}{2}\right) \quad\quad(i=1,\ldots,N).$$
	This proves readily that  $v_1$ is an eigenvector with eigenvalue $\lambda_1$ as claimed.
\end{proof}

We next consider the  $\lambda_k$ for $k>1$. 
We here do not present the  eigenvectors explicitely and prove a slightly weaker result:

\begin{lemma}\label{evk} For $k=2,\ldots, N$ there exist polynomials $p_k$ of order at most $k-2$, such that
	the vector
	$v_k:=\left(z_1^{k-1}\sqrt{1-z_1^2},\ldots,z_N^{k-1}\sqrt{1-z_N^2}\right)^T$
	satisfies
	$$\tilde Sv_k= \left((\lambda_k z_1^{k-1}+p_k(z_1)) \sqrt{1-z_1^2},\ldots,
	(\lambda_k z_N^{k-1}+p_k(z_N)) \sqrt{1-z_N^2}\right)^T.$$
\end{lemma}

\begin{proof}  
	We first consider the case  $k=2$. We here have
	\begin{align}
	(\tilde Sv_2)_i =& 4\sum_{ l\ne i} \frac{1-z_i^2}{(z_i-z_l)^2}z_i\sqrt{1-z_i^2}+ 2(a+b)\frac{1+z_i}{1-z_i}z_i\sqrt{1-z_i^2}\\
	& + 2b\frac{1-z_i}{1+z_i}z_i\sqrt{1-z_i^2} - 4\sum_{ l\ne i} \frac{z_l(1-z_l^2)\sqrt{1-z_i^2}}{(z_i-z_l)^2} 
\notag	\end{align}
Hence,
\begin{align}
	(\tilde Sv_2)_i
	=& 4\sqrt{1-z_i^2}\left( \sum_{ l\ne i} \frac{(1- z_i^2)z_i-(1- z_l^2)z_l}{(z_i-z_l)^2}+ \frac{a+b}{2}\frac{1+z_i}{1-z_i}z_i
	+ \frac{b}{2}\frac{1-z_i}{1+z_i}z_i\right).\notag\\
	=& 4\sqrt{1-z_i^2}\left( \sum_{ l\ne i} \frac{1-z_l^2-z_iz_l-z_i^2}{z_i-z_l}+ \frac{a+b}{2}\frac{1+z_i}{1-z_i}z_i
	+ \frac{b}{2}\frac{1-z_i}{1+z_i}z_i\right)\notag\\
	=& 4\sqrt{1-z_i^2}\Biggl( \sum_{ l\ne i}  \frac{1+z_l(z_i-z_l) +2z_i(z_i-z_l) -3z_i^2}{z_i-z_l}\notag\\
	&\quad\quad\quad\quad\quad\quad+ \frac{a+b}{2}\frac{1+z_i}{1-z_i}z_i
	+ \frac{b}{2}\frac{1-z_i}{1+z_i}z_i\Biggr)\notag\\
	=& 4\sqrt{1-z_i^2}\Biggl((c-z_i)+ 2z_i(N-1) +(1-3z_i^2) \sum_{ l\ne i} \frac{1}{z_i-z_l}
	\notag\\
	&\quad\quad\quad\quad\quad\quad+ \frac{a+b}{2}\frac{1+z_i}{1-z_i}z_i
	+ \frac{b}{2}\frac{1-z_i}{1+z_i}z_i\Biggr)\notag
	\end{align}
	with $c:=\sum_{j=1}^Nz_j$.A short computation and (\ref{lemma-formel1}) now lead to
	$$(\tilde Sv_2)_i = 4\sqrt{1-z_i^2}\left((c-z_i)+ 2z_i(N-1) + \frac{\alpha+1}{2}(2z_i+1) + \frac{\beta+1}{2}(2z_i-1)\right) $$
	for $i=1,\ldots,N$. 
	This proves readily that $\tilde Sv_2$ has the form as claimed in the lemma with some constant polynomial $p_2$.
	
	We now turn to the case $k\ge3$.
	We here have
	\begin{align}
	(\tilde Sv_k)_i =& 4\sum_{ l\ne i} \frac{1-z_i^2}{(z_i-z_l)^2}z_i^{k-1}\sqrt{1-z_i^2}+ 2(a+b)\frac{1+z_i}{1-z_i}z_i^{k-1}\sqrt{1-z_i^2}\\
	  & + 2b\frac{1-z_i}{1+z_i}z_i^{k-1}\sqrt{1-z_i^2} - 4\sum_{ l\ne i} \frac{z_l^{k-1}(1-z_l^2)\sqrt{1-z_i^2}}{(z_i-z_l)^2}; \notag
        \end{align}
        thus
        \begin{align}
	(\tilde Sv_k)_i
	=& 4\sqrt{1-z_i^2}\biggl( \sum_{ l\ne i} \frac{z_i^{k-1}-z_i^{k+1}-z_l^{k-1}+z_l^{k+1}}{(z_i-z_l)^2}+\notag\\
	&\quad\quad\quad\quad\quad\quad+ \frac{a+b}{2}\frac{1+z_i}{1-z_i}z_i^{k-1}
	+ \frac{b}{2}\frac{1-z_i}{1+z_i}z_i^{k-1}\Biggr)\notag
	\end{align}
	with
	\begin{align}
	z_i^{k-1}&-z_i^{k+1}-z_l^{k-1}+z_l^{k+1}= \notag\\
	&=(z_i-z_l)\Biggl( z_i^{k-2}+  z_i^{k-3}z_l + \ldots+  z_l^{k-2}-
	z_i^{k}-  z_i^{k-1}z_l - \ldots-  z_l^{k}\Biggr)\notag\\
	&=(z_i-z_l)\Biggl( (z_l-z_i)\Bigr(z_l^{k-3}+ 2z_l^{k-4}z_i+ \ldots+ (k-2)z_i^{k-3}\Bigr) +(k-1)z_i^{k-2}\notag\\
	&\quad\quad\quad\quad\quad- (z_l-z_i)\Bigr(z_l^{k-1}+ 2z_l^{k-2}z_i+ \ldots+ kz_i^{k-1}\Bigr) -(k+1)z_i^{k}\Biggr).
	\end{align}
	We thus conclude that 
	\begin{align}
	(\tilde Sv_k)_i =& 4\sqrt{1-z_i^2}\Biggl( \sum_{ l\ne i}\Bigl(z_l^{k-1}+ 2z_l^{k-2}z_i+ \ldots+ kz_i^{k-1}\notag\\
	&\quad\quad\quad\quad\quad\quad\quad\quad\quad\quad-
	z_l^{k-3}- 2z_l^{k-4}z_i- \ldots- (k-2)z_i^{k-3}\Bigr)\notag\\
	&+ \sum_{ l\ne i}\frac{(k-1)z_i^{k-2} -(k+1)z_i^{k}}{z_i-z_l}   
	+ \frac{a+b}{2}\frac{1+z_i}{1-z_i}z_i^{k-1}
	+ \frac{b}{2}\frac{1-z_i}{1+z_i}z_i^{k-1}\Biggr).
	\notag\end{align}
	With (\ref{lemma-formel1}), a suitable constant $C$, and with a suitable polynomials\\
	$q_k,r_k^{(1)},r_k^{(2)},r_k^{(3)},r_k^{(4)}$ of order at most $k-2$
	we thus obtain
	\begin{align}
	(\tilde Sv_k)_i 
	=&4\sqrt{1-z_i^2}\Biggl(C- z_i^{k-1}- 2z_i^{k-1}- \ldots- (k-1)z_i^{k-1} +k(N-1)z_i^{k-1} +q_k(z_i)\notag\\
	&\quad\quad\quad+ \frac{a+b}{2}\frac{z_i^{k}+z_i^{k-1}+ (k-1)z_i^{k-2}-(k+1)z_i^{k}}{1-z_i}\notag\\
	&\quad\quad\quad+ \frac{b}{2}\frac{z_i^{k-1}-z_i^{k}- (k-1)z_i^{k-2}+(k+1)z_i^{k}}{1+z_i}\Biggr)\notag\\
	=&4\sqrt{1-z_i^2}\Biggl(\Bigl(k(N-1)- \frac{(k-1)k}{2}\Bigr)z_i^{k-1}+r_k^{(1)}(z_i)\notag\\
	&\quad\quad\quad+ \frac{a+b}{2}(kz_i^{k-1}+r_k^{(2)}(z_i)) + \frac{b}{2}(kz_i^{k-1}+r_k^{(3)}(z_i))\Biggr)\notag\\
	=&\sqrt{1-z_i^2}\bigg(2k(2N+\alpha+\beta+1-k)z_i^{k-1}+r_k^{(4)}(z_i)\bigg).
	\end{align}
	This implies the lemma for $k\ge3$.
\end{proof}

\begin{proof}[Proof of Proposition \ref{ev-ew}]
Lemma \ref{ev1}, Lemma  \ref{evk},  induction on $k=1,\ldots,N$, and an obvious computation 
easily lead to the first statements of the Proposition for some polynomials $(q_k)_{k=0,\ldots,N-1}$ with $q_k$ of order at most $k$ for $k=0,....,N-1$.

It remains to identify the polynomials  $(q_k)_{k=0,\ldots,N-1}$ as  finite sequence of orthogonal polynomials w.r.t.
the discrete measure
\begin{equation}\label{orthogonality-measure}
\mu_{N,\alpha,\beta}:=(1-z_1^2)\delta_{z_1}+\ldots+(1-z_N^2)\delta_{z_N}.
\end{equation}
For this, consider a sequence of orthonormal polynomials $(q_{l}^{(\alpha,\beta)})_{l=0,\ldots,N-1}$
associated with $\mu_{N,\alpha,\beta}$
as  for instance in \cite{C}. We then have
\begin{align}\label{orthgonality-equation}
\sum_{i=1}^N q_{l}^{(\alpha,\beta)}(z_i) q_{k}^{(\alpha,\beta)}(z_i)(1-z_i^2)=\delta_{l,k} \quad\quad(k,l=0,\ldots,N-1).
\end{align}
This orthogonality fits to the fact that we may write the symmetric matrix $\tilde S$ as
$\tilde S=T^{-1}\cdot \operatorname{diag}(\lambda_1,\ldots,\lambda_N)\cdot T$ with some orthogonal matrix $T\in O(N)$.
We thus obtain that the  $q_k$ in Corollary \ref{ev-ew} are necessarily equal to the 
$q_{k}^{(\alpha,\beta)}$ up to normalization constants as claimed.
\end{proof}

We finally complete the proof of Theorems   \ref{theoremclttrig}  and \ref{theoremCLT}
and of Proposition \ref{proposition-jacobi-det}.
%We start with the proof of Proposition \ref{proposition-jacobi-det}.

\begin{proof}[Proof of Proposition \ref{proposition-jacobi-det}]
  It can be easily checked that the matrices $S$ and $\tilde S$ from Theorems   \ref{theoremclttrig}  and \ref{theoremCLT}
  are related by $\tilde{S}=DSD$,
  where the matrix \begin{align*}
  	D:=\text{diag}\left(\frac{1}{\sqrt{1-z_1^2}},...,\frac{1}{\sqrt{1-z_N^2}}\right)
  \end{align*} 
  is the Jacobi matrix of the inverse Transformation $T^{-1}$ of \ref{transformationtrigalgebraic}
 at the position $z=(z_1,...,z_N)$.
This and Proposition \ref{ev-ew} 
 ensure that
\begin{align*}
\det( S)\cdot \det( D)^2=\det(\tilde S)= 2^N \cdot N!\cdot (N+\alpha+\beta+1)_N.
\end{align*}
Furthermore, $\det(D)$ can be computed via
(\ref{prod 1-zi}) and (\ref{prod 1+zi}) which finally leads to the proof of Proposition \ref{proposition-jacobi-det}.
\end{proof}

\begin{proof}[Proof of Theorems \ref{theoremCLT} and \ref{theoremclttrig}]
Proposition \ref{proposition-jacobi-det} ensures that the measure with the density (\ref{Limit})
is in fact a probability measure and hence  the normal distribution $\mathcal{N}(0,\Sigma)$.
As a consequence of the first step of the proof, we conclude that
$\sqrt{\kappa}(X_\kappa-z)$ converges in distribution
to the  normal distribution as claimed in Theorem \ref{theoremCLT}.

The Delta-method for the central limit theorem of random variables (see Section  3.1 of \cite{vV})
now immediatly yields Theorem \ref{theoremclttrig}.
\end{proof}
\section*{Acknowledgement}
Kilian Hermann was supported by the Deutsche Forschungsgemeinschaft (DFG) via RTG 2131
High-dimensional
Phenomena in Probability, Fluctuations and Discontinuity.


\begin{thebibliography}{}
	
\bibitem[AHV]{AHV} S. Andraus, K. Hermann, M. Voit,
  Limit theorems and soft edge of freezing random matrix models via dual orthogonal polynomials. Preprint, arXiv:2009.01418.

	
	\bibitem[AKM1]{AKM1} S. Andraus, M. Katori, S. Miyashita, Interacting particles on the line 
	and Dunkl intertwining operator of type $A$: Application to the freezing regime. 
	\textit{J. Phys. A: Math. Theor. } 45  (2012) 395201.
	
	\bibitem[AKM2]{AKM2} S. Andraus, M. Katori, S. Miyashita, 
	Two limiting regimes of interacting Bessel processes. 
	\textit{J. Phys. A: Math. Theor. } 47  (2014) 235201.
	
\bibitem[AV1]{AV1} S. Andraus, M. Voit, Limit theorems
 for multivariate Bessel processes in the freezing regime. 
\textit{Stoch. Proc. Appl. } 129 (2019), 4771-4790.

	
\bibitem[AV2]{AV2} S. Andraus, M. Voit, Central limit theorems
 for multivariate Bessel processes in the freezing regime II: the covariance matrices of the limit.
 \textit{J. Approx. Theory}  246 (2019), 65-84.



	\bibitem[BG]{BG} A. Borodin, V. Gorin, General $\beta$-Jacobi corners process and the Gaussian free field.
		\textit{Comm. Pure Appl. Math. } 68 (2015),  1774-1844.
	
	\bibitem[C]{C} T.S. Chihara, An Introduction to Orthogonal Polynomials. Gordon and Breach, New York 1978.
	
	
	\bibitem[Dem]{Dem} N. Demni,
	$\beta$-Jacobi processes.
	\textit{Adv. Pure Appl. Math.} 1 (2010), 325-344.
	
	

	
	\bibitem[DV]{DV} J.F. van Diejen, L. Vinet, Calogero-Sutherland-Moser Models.
	CRM Series in Mathematical Physics, Springer-Verlag 2000.
	
	\bibitem[DE1]{DE1} I. Dumitriu, A. Edelman, Matrix models for beta-ensembles. \textit{ J. Math. Phys.} 43 (2002),  5830-5847.
	
	\bibitem[DE2]{DE2} I. Dumitriu, A. Edelman, Eigenvalues of Hermite and Laguerre ensembles: large beta asymptotics,
	\textit{Ann. Inst. Henri Poincare (B)} 41 (2005), 1083-1099.
	
	\bibitem[F]{F}  P. Forrester, Log Gases and Random Matrices, London Mathematical Society, London, 2010.
	
	\bibitem[FW]{FW} P. Forrester, S.  Warnaar, The importance of the Selberg integral, \textit{ Bull. Amer. Math. Soc.} 45 (2008),
	489--534.
      \bibitem[GK]{GK} V. Gorin, V. Kleptsyn, Universal objects of the infinite beta random matrix theory. arXiv preprint arXiv:2009.02006.
        
	\bibitem[HS] {HS} G. Heckman, H. Schlichtkrull, Harmonic Analysis and Special
	Functions on Symmetric Spaces; 
	Perspectives in Mathematics, vol. 16, Academic Press, California, 1994.

        

\bibitem[I]{I} M. Ismail,  Classical and quantum orthogonal polynomials in one variable.
Cambridge University Press, Cambridge 2005. 

	
	\bibitem[J]{J} T. Jiang, Limit theorems for beta Jacobi ensembles. 
	\textit{Bernoulli} 19 (2013), 1028--1046.
	
	\bibitem[K]{K} R. Killip,
	Gaussian fluctuations for $\beta$ ensembles.
	\textit{ Int. Math. Res. Not.} 2008, no. 8, Art. ID rnn007, 19 pp..
	
	\bibitem[KN]{KN} R. Killip, I. Nenciu,  Matrix models for circular ensembles.
	\textit{ Int. Math. Res. Not.} 50 (2004), 2665--2701. 
	
	\bibitem[L]{L} R.A. Lippert, A matrix model for the $\beta$-Jacobi ensemble,
	\textit{J. Math. Phys.} 44 (2003), 4807-4816.
	
	\bibitem[M]{M} M. Mehta, Random matrices (3rd ed.), Elsevier/Academic Press, Amsterdam, 2004. 
	
	\bibitem[N]{N} J. Nagel, 
	Nonstandard limit theorems and large deviations for the Jacobi beta
	ensemble.
	\textit{Random Matrices Theory Appl.} 3 (2014),  Article ID 1450012.
	
	\bibitem[S]{S} G. Szeg{\"o}, Orthogonal Polynomials. 
	Colloquium Publications, American Mathematical Society, Providence, 1939.
	
	\bibitem[vV]{vV} A.W. van der Vaart, Asymptotic Statistics, Cambridge University Press, 1998.
	
	\bibitem[V]{V} M. Voit,  Central limit theorems for multivariate Bessel processes in the freezing regime. \textit{ J. Approx. Theory } 239 (2019), 210--231.
	
	
 \bibitem[VW]{VW} M. Voit, J.H.C. Woerner, Functional 
 central limit theorems for multivariate Bessel processes in the
 freezing regime. \textit{ Stoch. Anal. Appl.}, https://doi.org/10.1080/07362994.2020.1786402, arXiv:1901.08390. 
	
	
\end{thebibliography}
\end{document}